\documentclass[a4paper,10pt]{amsart}

\pdfoutput=1
\usepackage[
            pagebackref  = true,
            colorlinks   = true,
            pdfauthor    = {Giacomo Cherubini},
            pdftitle     = {Almost periodic functions and hyperbolic counting},
            pdfkeywords  = {},
            pdfcreator   = {Pdflatex},
            pdfpagemode  = UseNone,
            pdfstartview = FitH
           ]
           {hyperref}
\usepackage{xcolor}
\usepackage{pdfpages}
\usepackage[normalem]{ulem}
\usepackage{url}
\usepackage{mathrsfs}
\usepackage{csquotes}
\usepackage{mathtools}
\renewcommand{\MR}[1]{}

\numberwithin{equation}{section}

\newcommand\CC{\mathbb{C}}
\newcommand\HH{\mathbb{H}}
\newcommand\NN{\mathbb{N}}
\newcommand\QQ{\mathbb{Q}}
\newcommand\RR{\mathbb{R}}
\newcommand\ZZ{\mathbb{Z}}
\newcommand\eps{\varepsilon}
\newcommand\vol{\mathrm{vol}}
\newcommand\PSL{\mathrm{PSL}}

\theoremstyle{plain}
\newtheorem{thm}{Theorem}[section]
\newtheorem{theorem}[thm]{Theorem}
\newtheorem{lemma}[thm]{Lemma}
\newtheorem{corollary}[thm]{Corollary}
\newtheorem{prop}[thm]{Proposition}

\theoremstyle{definition}
\newtheorem{definition}[thm]{Definition}
\newtheorem{remark}[thm]{Remark}

\begin{document}

\title[Almost Periodic Functions and Hyperbolic Counting]{Almost Periodic Functions and\\ Hyperbolic Counting}

\author{Giacomo Cherubini}

\address{
         Dipartimento di Matematica, Universit\`a di Genova,
         Via Dodecaneso 35, 16146 Genoa Italy
        }
\email{cherubini@dima.unige.it}

\begin{abstract}
We prove the existence of asymptotic moments and an estimate on the tails
of the limiting distribution for a specific class of almost periodic functions.
Then we introduce the hyperbolic circle problem, proving an estimate on the
asymptotic variance of the remainder that improves a result of Chamizo.
Applying the results of the first part we prove the existence of
limiting distribution and asymptotic moments for three functions
that are integrated versions of the remainder,
and were considered originally (with due adaptations to our settings)
by Wolfe, Phillips and Rudnick, and Hill and Parnovski.
\end{abstract}

\maketitle


\section{Introduction}

A classical problem in analytic number theory is that of determining
if a given function, arising from a number theoretical question,
admits a limiting distribution,
and, similarly, if it admits finite asymptotic moments.
As an example, consider the summatory von Mangoldt function
\begin{equation}\psi(x)=\sum_{1\leq n\leq x} \Lambda(n).\end{equation}                 
It is an old result of Wintner \cite{wintner_asymptotic_1935}
that, under the assumption of the Riemann hypothesis, the normalized remainder
\begin{equation}\label{apf01:eq006}
q(y) = \frac{\psi(e^y)-e^y}{e^{y/2}}
\end{equation}
admits finite asymptotic moments of every order, and it has a limiting distribution.
(the finiteness of the second moment had already been proved by Cram\'er \cite{cramer_mittelwertsatz_1922}).
Similarly, let $R(x)$ and $D(x)$ be the counting functions in the
Gauss circle problem and the Dirichlet divisor problem, namely
\begin{equation}
R(x) = \sum_{1\leq n\leq x} r(n),\qquad
D(x) = \sum_{1\leq n\leq x} d(n),
\end{equation}
where $r(n)$ is the number of ways of writing $n$ as
a sum of two squares, and $d(n)$ the number of divisors of $n$.
Then it is known that the normalized remainders
\begin{equation}
u(y)      = \frac{R(y^2)-\pi y^2}{y^{1/2}},\qquad
v(y) = \frac{D(y^2)-(y^2\log y^2 -(2C-1)y^2)}{y^{1/2}}
\end{equation}
admit asymptotic moments of order $1\leq k\leq 9$, and a limiting distribution
(here $C$ is Euler's constant).
The finiteness of the second moment is again a result of Cram\'er \cite{cramer_uber_1922},
the third and fourth moments can be found in the work of Tsang \cite{tsang_higher-power_1992},
while the existence of the moments of order up to nine, and of the limiting distribution,
was proved by Heath-Brown \cite{heath-brown_distribution_1991,heath-brown_dirichlet_1993}.

This paper is inspired by a recent article by
Akbary, Ng, and Shahabi~\cite{akbary_limiting_2014},
where the authors prove the existence of limiting distribution for a class
of almost periodic functions.
Here we prove a sufficient condition for the existence
of asymptotic moments for a similar class of functions,
complementing in this way Theorem 1.2 in \cite{akbary_limiting_2014}.

We have in mind applications to counting problems
in the hyperbolic plane (see \S\ref{apf01:subsection01}),
where also inputs from spectral theory of automorphic forms are needed.
These problems turn out to be much harder than their classical analogues,
due to the fact that we cannot exploit the spectral side
in such an effective way as done in the classical setting.
One of the reasons is that the spectrum of the euclidean Laplacian
is completely explicit, but we do not know the exact location of the spectrum
of the hyperbolic Laplacian. This complicates the proofs,
and prevents us from proving results as good as in the classical case
(at least with the same methods).
Since the functions that appear in these problems share many similarities
with more general almost periodic functions,
it seems appropriate to prove the existence of limiting distribution
and asymptotic moments in such a general framework.

The remainders in the prime number theorem,
the Gauss circle problem, and the Dirichlet divisor problem,
are suitable examples to describe the class of functions
that we consider in Definition \ref{apf_01:def001} below.
Let $q(y)$ be the normalized remainder in the prime number theorem,
as defined in \eqref{apf01:eq006}.
Assuming the Riemann hypothesis, there exist $y_0,X_0>0$ such that we can write
(see e.g. chapter 17 in \cite{davenport_multiplicative_2000}),
for $y\geq y_0$ and $X\geq X_0$,
\begin{equation}\label{apf01:eq001}
q(y)
=
-2\Re\bigg(\sum_{0<\gamma\leq X}\frac{1}{\rho}e^{i\gamma y}\bigg)
+
O\left(\frac{e^{y/2}\log(e^yX)}{X} + ye^{-y/2}\right)
\end{equation}
where $\rho=1/2+i\gamma$ are the non-trivial zeros of the Riemann zeta function.
For the Gauss circle problem (the Dirichlet divisor problem is similar),
one can show that there exist $y_0,X_0>0$ such that for $y\geq y_0$ and $X\geq X_0$ we can write
\begin{equation}\label{apf01:eq002}
u(y)
=
2\Re\bigg(\sum_{1\leq n\leq X} \frac{r(n)e^{-i\pi/4}}{2\pi\sqrt{2}\,n^{3/4}}e^{4\pi iy\sqrt{n}}\bigg)
+
O\left(\frac{X^\eps}{y^{1/2}}+\frac{y^\eps}{X^{1/2}}\right),
\end{equation}
For a reference see e.g. the book of Titchmarsh \cite[(12.4.4)]{titchmarsh_theory_1986}
or that of Ivi\'c \cite[Ch. 13]{ivic_riemann_1985}).

Identities \eqref{apf01:eq001} and \eqref{apf01:eq002} show that we can approximate
the functions $q$ and $u$
by finite linear combinations of complex exponentials. The coefficients
and the frequencies of such linear combinations depend on the problem.

\begin{definition}\label{apf_01:def001}
Let $p\in\RR$, $p\geq 1$, and let $\phi:[0,\infty)\to\RR$.
We say that $\phi$ is a \emph{$p$-function} if there exist a strictly increasing
sequence $\{\lambda_n\}$ of positive real numbers tending to infinity,
a sequence of complex numbers $\{r_n\}$, and numbers $y_0,X_0>0$, such that
$\phi\in L^p([0,y_0])$, and for $y\geq y_0$, $X\geq X_0$, we have
\begin{equation}\label{0308:eq001}
    \phi(y)
    =
    2\Re\Big(\sum_{0<\lambda_n\leq X} r_n e^{i\lambda_n y}\Big)
    +
    \mathcal{E}(y,X),
\end{equation}
where $\mathcal{E}(y,X)$ satisfies
\begin{equation}\label{ch0301:eq003}
    \lim_{Y\to\infty} \frac{1}{Y}\int_{y_0}^Y |\mathcal{E}(y,X(Y))|^pdy = 0
\end{equation}
for some non-decreasing function $X(Y)$ tending to infinity.
We say moreover that $\phi$ is a \emph{$(p,\beta)$-function} if $\phi$
is a $p$-function, and there exists $\beta\in\RR$ such that
\begin{equation}\label{ch0301:eq004}
    \sum_{T\leq \lambda_n\leq T+1} |r_n| \ll \frac{1}{T^\beta}\quad\text{for }T\gg 1.
\end{equation}
\end{definition}

\begin{remark}
It follows from \cite[Th. 1.2]{akbary_limiting_2014}
that a $(p,\beta)$-function $\phi$ such that $p\geq 2$ and $\beta>1/2$
is a $B^2$-almost periodic function.
Definition \ref{apf_01:def001} is slightly more general,
as we can consider $1\leq p<2$ and $\beta\leq 1/2$.
However, in Theorem \ref{ch0304:thm01}, Theorem \ref{ch0303:thm01},
and  Corollary \ref{apf01:cor01} below, one should keep in mind that
the results concern almost periodic functions as soon as $p\geq 2$ and $\beta>1/2$.
\end{remark}

The following Theorem \ref{ch0304:thm01} gives a sufficient condition to ensure that
a $(p,\beta)$-function admits finite asymptotic moments.
The condition
\begin{equation}\label{ap:eq005}
   \limsup_{Y\to\infty} \sup_{y\in[y_0,Y]} |\mathcal{E}(y,X(Y))| = 0
\end{equation}
is required to cover some cases where \eqref{ch0301:eq003} does not suffice.
Observe that \eqref{ap:eq005} implies \eqref{ch0301:eq003},
for any $p\geq 1$, since we have
\begin{equation}
   \frac{1}{Y}\int_{y_0}^Y |\mathcal{E}(y,X(Y))|^pdy
   \leq
   \sup_{y\in[y_0,Y]} |\mathcal{E}(y,X(Y))|^p.
\end{equation}

\begin{theorem}\label{ch0304:thm01}
Let $p\in\NN$, and let
$\phi$ be a $(p,\beta)$-function with $\beta>1-1/p$.
If $p$~is odd and greater than $1$, assume that \eqref{ap:eq005} holds.
Then the asymptotic moment
\begin{equation}\label{ch0304:eq062}
    L_n=\lim_{Y\to\infty}\frac{1}{Y}\int_0^Y \phi(y)^n dy\nonumber
\end{equation}
exists, for every $1\leq n\leq p$.
\end{theorem}

\begin{remark}
It is easy to see that the function $q(y)$ in \eqref{apf01:eq001} (under RH), and the function
$u(y)$ in \eqref{apf01:eq002}, are $p$-functions for every $p\geq 1$ (by choosing
$X(Y)=e^Y$ in the first case, and $X(Y)=Y$ in the second).
Because of the asymptotic formula for $N(T)$, the number of non-trivial zeros
of the Riemann zeta function with imaginary part in $(0,T]$, we have
(see Davenport \cite[p. 59]{davenport_multiplicative_2000})
\[
N(T) = \frac{T}{2\pi} \log\frac{T}{2\pi} - \frac{T}{2\pi} + O(\log T),
\]
and in view of the standard asymptotic formula
\[
\sum_{1\leq n\leq x} r(n) = \pi x+O(\sqrt{x}),
\]
we conclude that we have
\[
\sum_{\rho=1/2+i\gamma:\atop T\leq \gamma\leq T+1} \frac{1}{|\rho|}\ll \frac{\log T}{T}
\quad\text{and}\;
\sum_{T\leq \sqrt{n}\leq T+1} \frac{r(n)}{n^{3/4}} \ll \frac{1}{T^{1/2}}.
\]
This shows that the function $q$ is (under RH) a $(p,\beta)$-function
for $p\geq 1$ and any $\beta<1$, whereas the function $u$ is a $(p,\beta)$-function
for $p\geq 1$ and $\beta=1/2$.
\end{remark}

\begin{remark}
Notice that if $\phi$ is a $(p,\beta)$-function,
then it is also a $(p',\beta')$-function, for every 
$p'\leq p$ and $\beta'\leq \beta$.
It is thus interesting to find the largest $p$ and $\beta$ such
that $\phi$ is a $(p,\beta)$-function:
the pair $(p,\beta)$ gives, by Theorem \ref{ch0304:thm01},
the largest range for which we can prove finiteness of the moments of $\phi$.
\end{remark}
\begin{remark}
The existence of the second moment
was proved by Akbary et al. \cite[Th. 1.14]{akbary_limiting_2014}
for $(2,\beta)$-functions with $\beta>1/2$,
as a corollary of the fact that such functions
are Besicovitch $B^2$-almost periodic.
Theorem \ref{ch0304:thm01} extends their result
to moments of higher order.
\end{remark}
\begin{remark}
We give an expression for $L_n$ in \eqref{ch0304:eq051},
with notation introduced in section \ref{section02}.
The value of the moments is obtained as the sum of an absolutely
convergent series that involves the data from the sequence of the
frequencies $\lambda_n$ and from the coefficients~$r_n$.
\end{remark}
\begin{remark}
If $\phi$ is a $(p,\beta)$-function for $p$ arbitrarily large
and every $\beta<1$,
then Theorem \ref{ch0304:thm01}
implies that all the moments of $\phi$ exist.
This is the case of the remainder $q(y)$ in the prime number theorem
(under RH).

On the other hand, when $\beta\leq 1/2$, an extra input must be given
in order to show that the second and higher moments exist.
This is what happens for the remainder $u(y)$ in the Gauss circle problem,
where additional properties of the problem are used to prove
the existence of asymptotic variance and moments of order $1\leq n\leq 9$.
Indeed, if $\phi$ is a $(p,\beta)$-function with $\beta\leq 1/2$,
then it does not follow from \cite[Th. 1.2]{akbary_limiting_2014}
that $\phi$ is an almost periodic function, and some extra
information must be used to show if this is the case.

We consider the cases $\beta=1/2$ and $\beta=1$ as \enquote{extremal},
so that Theorem~\ref{ch0304:thm01} covers the intermediate
situation when $\beta$ is some number strictly between $1/2$ and~$1$.
\end{remark}

\begin{remark}
Consider an irreducible unitary cuspidal automorphic representation $\pi$
of $\mathrm{GL}_d(\mathbb{A}_\QQ)$, and let $L(s,\pi)$ be the automorphic
$L$-function attached to $\pi$.

Let $\psi(x,\pi)$ be the prime counting function associated to $L(s,\pi)$,
$M(x,\pi)$ the main term in the asymptotic expansion of $\psi(x,\pi)$,
and denote the remainder by $E(x,\pi)$.
Assuming the corresponding Riemann hypothesis for $L(s,\pi)$, the function
\[e^{-y/2}E(e^y,\pi)\]
has the structure of a $(p,\beta)$-function
for $p\geq 1$ and any $\beta<1$ (see \cite[Proposition~4.2]{akbary_limiting_2014}).
Theorem \ref{ch0304:thm01} shows that $e^{-y/2}E(e^y,\pi)$ admits
finite moments of every order.
The result should be compared with \cite[Cor. 1.15]{akbary_limiting_2014},
where the existence of the variance is proved.
\end{remark}

If we consider a $(p,\beta)$-function $\phi$ such that
$p\geq 2$ and $\beta>1/2$,
it follows from \cite[Th. 1.2]{akbary_limiting_2014}
that $\phi$ admits a limiting distribution.
Our second result concerns the tails of the limiting distribution.

\begin{theorem}\label{ch0303:thm01}
Let $\phi$ be a $(p,\beta)$-function with $p\geq 2$ and $\beta>1/2$.
Then $\phi$ admits a limiting distribution $\mu$ with tails of size
\begin{equation}\label{ch0303:eq019}
    \mu((-\infty,-S]\cup[S,+\infty))\ll S^{-(2\beta-1)/(2-2\beta)}.
\end{equation}
For $\beta=1$ we have exponential decay, that is, there exists a
positive constant $c_\phi>0$ such that
\(\mu((-\infty,-S]\cup[S,+\infty))\ll \exp(-c_\phi S).\)
For $\beta>1$ the measure $\mu$ is compactly supported.
\end{theorem}

\begin{remark}
The proof of the existence of the measure $\mu$ can be found in several papers
(see e.g.
\cite{heath-brown_distribution_1991,bleher_distribution_1993,rubinstein_chebyshevs_1994,akbary_limiting_2014}).
The estimate on the tails for general $\beta$ does not seem however to appear in these papers.
In section \ref{apf03:section03} we prove Theorem~\ref{ch0303:thm01},
following the argument of \cite[p. 178-181]{rubinstein_chebyshevs_1994},
which easily generalizes to \eqref{ch0303:eq019}.
\end{remark}

\begin{remark}
In the case of the prime number theorem and $L$-functions discussed
in \cite{rubinstein_chebyshevs_1994}, the remainder terms
have an almost periodic expansion with corresponding 
coefficients $r_n$ that satisfy
\begin{equation}\label{ch0303:eq033}
    \sum_{T\leq \lambda_n\leq T+1} |r_n| \ll \frac{\log T}{T}.
\end{equation}
This leads to exponential decay of type $O(\exp(-c\sqrt{\lambda}))$
for the tails of the limiting distributions.
Similarly, a bound in \eqref{ch0303:eq033} of type
$O(T^{-1}\log^{m}(T))$, $m\geq 0$, leads to an upper bound
for the tails of $\mu$ of type $O(\exp(-c\lambda^{1/(m+1)}))$.
\end{remark}

\begin{remark}
Lower bounds for the tails of the limiting distribution $\mu$ of a general
$(p,\beta)$-function $\phi$ can also be proved by similar argument as in \cite{rubinstein_chebyshevs_1994},
but we have decided not to discuss them here.
Moreover, assuming the extra condition that the frequencies $\lambda_n$
are linearly independent, a much stronger decay
on the tails of $\mu$ can be proved, and one can show that the Fourier transform
of $\mu$ can be expressed in terms of Bessel functions
(see \cite[Th. 1.9]{akbary_limiting_2014} and also \cite{montgomery_zeta_1980,montgomery_large_1987,rubinstein_chebyshevs_1994}).
\end{remark}

If a function $\phi$ admits a limiting distribution $\mu$
and finite asymptotic moments, a natural question to ask is whether the moments
of $\phi$ coincide with the moments of the distribution $\mu$.
This can be a nontrivial question, as there exist functions
for which the moments do not agree (for instance if the function has very rare but very large peaks).

Due to the estimate on the tails of the limiting distribution
provided by Theorem \ref{ch0303:thm01}, we can prove the following corollary,
where we show that up to a certain order the moments
of a $(p,\beta)$-function are indeed equal to those of its limiting distribution.

\begin{corollary}\label{apf01:cor01}
Let $\phi$ be a $(p,\beta)$-function with $p\geq 2$ and $\beta>1-1/(2p+4)$.
Then the moments of $\phi$ of order $1\leq n\leq p$ coincide with the moments
of its limiting distribution $\mu$. In other words, for $1\leq n\leq p$ we have
\begin{equation}\label{ch0304:eq069}
    \lim_{Y\to\infty}\frac{1}{Y}\int_0^Y\phi(y)^ndy=\int_\RR x^n d\mu.
\end{equation}
\end{corollary}


\subsection{Applications to the hyperbolic circle problem}\label{apf01:subsection01}


In the second part of the paper we apply the above results to
the hyperbolic lattice point counting problem, which is defined as follows.
For $\Gamma\leq\PSL(2,\RR)$ a cofinite Fuchsian group
and $z,w\in\HH$, define the function
\begin{equation}\label{ch02:def:N}
    N(s,z,w)
    =
    \{\gamma\in\Gamma:\;d(z,\gamma w)\leq s\},
\end{equation}
where $d$ is the hyperbolic distance.
The function $N(s,z,w)$ counts the number of translates of the point $w$ by elements
$\gamma$ of the group~$\Gamma$ that fall inside the hyperbolic ball $B(z,s)$
of center $z$ and radius $s$.

Spectral theory of automorphic forms provides the main asymptotic
of $N(s,z,w)$ as $s$ tends to infinity,
as well as a finite number of secondary terms (associated to the small eigenvalues),
that we collect in the \enquote{complete main term} $M(s,z,w)$, see \eqref{ch02:def:mainterm}.
The remainder in the problem is then defined as
\begin{equation}
   E(s,z,w) = N(s,z,w) - M(s,z,w).
\end{equation}

In section \ref{apf_04_section04} we introduce the notation and technical estimates
related to the problem, and we prove an upper bound on the asymptotic variance for the
normalized remainder term $e(s,z,w):=e^{-s/2}E(s,z,w)$ as follows.

\begin{theorem}\label{ch00:intro:thm01}
Let $\Gamma$ be a cofinite Fuchsian group, let $z,w\in\HH$,
and let $T\gg 1$. Then
\begin{equation}\label{ch00:intro:eq004}
   \int_T^{T+1} |e(s,z,w)|^2 ds \ll T.
\end{equation}
\end{theorem}

\begin{remark}
Theorem~\ref{ch00:intro:thm01} improves on a result by Chamizo \cite[Cor. 2.1.1]{chamizo_applications_1996}
that corresponds, in our notation, to an upper bound $O(T^2)$ in \eqref{ch00:intro:eq004}.
The method of proof is different: Chamizo's proof uses the large sieve
in Riemann surfaces \cite{chamizo_large_1996}, while we directly integrate in the pretrace formula for $e(s,z,w)$.
This strategy was suggested in \cite[p. 27]{chamizo_topics_1994}.
Moreover, differently from Chamizo, we include the contribution associated to
the eigenvalue $\lambda=1/4$ in the main term $M(s,z,w)$.
\end{remark}
\begin{remark}
The problem of the finiteness of the asymptotic second moment of $e(s,z,z)$
was already discussed by Phillips and Rudnick
\cite{phillips_circle_1994}. Like us, they were not able to show that the second moment
of the normalized remainder is finite. On the other hand, they show that the limit
\[\lim_{T\to\infty}\frac{1}{T}\int_0^T |e(s,z,z)|^2ds\]
is non-zero, and they provide numerics that suggest that the limit should exist.
\end{remark}
\begin{remark}
Theorem \ref{ch00:intro:thm01} shows that the asymptotic variance of $e(s,z,w)$,
if not finite, diverges at most linearly with $T$.
Also, it shows that the function $E(s,z,w)$ is \enquote{on average} bounded
by $O(\sqrt{s}e^{s/2})$, which is consistent with the conjectural bound
\[E(s,z,w)=O\big(e^{s(1/2+\eps)}\big),\quad\forall\;\eps>0.\]
A straightforward consequence of Theorem~\ref{ch00:intro:thm01}
is that for every $\alpha>0$ the set
\begin{equation}\{s\geq 0\;:\; |E(s,z,w)| > s^{1+\alpha}e^{s/2}\}\end{equation}
has finite Lebesgue measure, hence $E(s,z,w)$ violates the conjectural bound at most in a set of finite measure.
\end{remark}

We move then to consider certain integrated versions of $E(s,z,w)$.
We start by defining, for $\Gamma$ cocompact, $z\in\Gamma\backslash\HH$, and $s\geq 0$,
the following functions:
\begin{equation}
\begin{gathered}
    G_1(s,z) := \frac{1}{e^s}\int_{\Gamma\backslash\HH} |E(s,z,w)|^2 d\mu(w),\\
    G_2(s)   := \frac{1}{e^s}\iint_{\Gamma\backslash\HH\times\Gamma\backslash\HH} |E(s,z,w)|^2 d\mu(z) d\mu(w).
\end{gathered}\nonumber
\end{equation}
The functions $G_1,G_2$ are thus defined by integrating away one
(resp. both) space variable.
For $\Gamma$ cocompact and $z\in\Gamma\backslash\HH$, $s\geq 0$,
we define also
\begin{equation}
    G_3(s,z) := \frac{1}{e^{s/2}} \int_0^s E(x,z,z) dx,\nonumber
\end{equation}
which gives the integral of $E(s,z,z)$ in the radial variable.
Applying the results of the first part of the paper we obtain the following theorem.


\begin{theorem}\label{ch0305:thm003}
Let $\Gamma$ be cocompact, $z\in\Gamma\backslash\HH$, and let $s\geq 0$.
Then for $i=1,2,3$ the function $G_i$ is bounded in $s$,
admits moments of every order, limiting distribution $\mu_i$ of compact support,
and the moments of $G_i$ coincide with the moments of $\mu_i$.
\end{theorem}


\begin{remark}
Notice that, since the functions $G_1$ and $G_3$ depend on $z$,
the limiting distributions $\mu_1$ and $\mu_3$ will also depend on $z$.
\end{remark}
\begin{remark}
The function $G_1$ was considered by Hill and Parnovski \cite{hill_variance_2005},
who only proved that $G_1$ is bounded.
The integration of both space variables that defines $G_2$
was studied by Wolfe \cite{wolfe_asymptotic_1979}, who again
studied its pointwise behaviour
but did not have the distributional result.
The function $G_3$ was studied by Philllips and Rudnick \cite{phillips_circle_1994}
in relation to the proof that the asymptotic mean of $E(s,z,z)$ vanishes.
They also consider cofinite groups that are not cocompact, and
it is probably possible
to extend the proof for $G_3$ to the general cofinite case,
but we have refrained from doing this here.
\end{remark}
\begin{remark}
Whereas we cannot prove that the remainder $e(s,z,z)$ admits finite variance,
we see that the function $G_3(s,z)$
admits not only finite variance, but finite moments of every order.
In \cite{cherubini_variance_2015} Cherubini and Risager
considered integration of $e(s,z,z)$ to fractional order $0<\alpha<1$,
and showed that the resulting function $e_\alpha(s,z)$,
has finite asymptotic variance for every $0<\alpha<1$.
This means that not a full integration, but integration
to any positive small order suffices to give finite second moment.
From the variance of $e_\alpha(s,z)$ one might then expect to recover,
in the limit as $\alpha$ tends to $0$,
the variance of $e(s,z,z)$, but it seems hard to prove this.
We refer to \cite[\S 8]{cherubini_variance_2015}
for a more detailed discussion on the subject.
\end{remark}


\section{Asymptotic moments}\label{section02}

In this section we prove Theorem \ref{ch0304:thm01}.
Let $\phi$ be a $(p,\beta)$-function.
Set $\lambda_{-n}=-\lambda_n$ and $r_{-n}=\overline{r}_n$,
and define the function
\begin{equation}\label{ch0302:eq008}
    S(y,X):=\sum_{n\in\ZZ:\atop|\lambda_n|\leq X}r_ne^{i\lambda_n y},
\end{equation}
so that we have, for $y\geq y_0$ and $X\geq X_0$, the identity
\begin{equation}\label{ch0302:eq009}
    \phi(y)=S(y,X)+\mathcal{E}(y,X).\nonumber
\end{equation}
In view of this relation and of assumption \eqref{ch0301:eq003},
we expect that the moments
of $\phi$ can be computed simply by looking at the moments of
$S(y,X)$.

On taking the $n$-th power we can write
\begin{equation}\label{report:eq001}
S(y,X)^n = \sum_{J\in\ZZ^n:\atop |\lambda_J|\leq X} A(r_J)e^{iy\vartheta(\lambda_J)},
\end{equation}
where $J=(j_1,\ldots,j_n)$ is a multiindex, $\lambda_J=(\lambda_{j_1},\ldots,\lambda_{j_n})$
(similarly $r_J=(r_{j_1},\ldots,r_{j_n})$), $|\lambda_J|\leq X$ means that all entries are smaller than $X$, and
\begin{equation}\label{1707:eq001}
A(r_J)=\prod_{k=1}^n r_{j_k},\quad \vartheta(\lambda_J)=\sum_{k=1}^n \lambda_{j_k}.
\end{equation}

Note that the sum in \eqref{report:eq001} runs over $J\in\ZZ^n$; however, in order to simplify notation, we will omit the condition $J\in\ZZ^n$, and only write $|\lambda_J|\leq X$, throughout the rest of this section.
Setting $X=X(Y)$ and integrating in $y$ we get
\begin{equation}\label{ch0304:eq031}
    \frac{1}{Y}\int_{y_0}^YS(y,X(Y))^n dy
    =
    \sum_{|\lambda_J|\leq X(Y)}
    A(r_J)
    \frac{1}{Y}\int_{y_0}^Y e^{iy\vartheta(\lambda_J)}dy.
\end{equation}
The integral equals $(Y-y_0)$ when $\vartheta(\lambda_J)=0$, and otherwise we can bound
\begin{equation}\label{ch0304:eq032}
    \bigg|\frac{1}{Y}\int_{y_0}^Y e^{iy\vartheta_g(\lambda_J)}dy\bigg|
    \leq
    \frac{4}{1+Y|\vartheta_g(\lambda_J)|}.
\end{equation}
We define the sets (diagonal and off-diagonal)
\begin{equation}\label{ch0304:eq033}
    \begin{aligned}
        \mathscr{D}^Y
        &:=
        \{|\lambda_J|\leq X(Y)\;|\;\vartheta(\lambda_J)=0\}
        \\
        \mathscr{O}^Y
        &:=
        \{|\lambda_J|\leq X(Y)\;|\;\vartheta(\lambda_J)\neq 0\}.
    \end{aligned}\nonumber
\end{equation}
The contribution of the off-diagonal is negligible if we assume that the coefficients $r_n$
decay sufficiently fast, and we have the following.

\begin{prop}\label{ch0304:prop01}
Let $S(y,X)$ be as in \eqref{ch0302:eq008}, and assume that
\eqref{ch0301:eq004} holds with $\beta>1-1/n$. Then the limit
\begin{equation}\label{ch0304:eq050}
    \lim_{Y\to\infty}\frac{1}{Y}\int_{y_0}^Y S(y,X(Y))^n dy
\end{equation}
exists, and is given by
\begin{equation}\label{ch0304:eq051}
    \sum_{\lambda_J\in\mathscr{D}^\infty} A(r_J),
\end{equation}
where $A(r_J)$ is defined in \eqref{1707:eq001}.
The sum in \eqref{ch0304:eq051} is absolutely convergent, as shown at the end of the proof.
\end{prop}

\begin{proof}
From \eqref{ch0304:eq031} and \eqref{ch0304:eq032} we can write
\begin{equation}\label{ch0304:eq052}
\begin{split}
    \frac{1}{Y}\int_{y_0}^Y S(y,X(Y))^n dy
    &=
    \sum_{\lambda_J\in \mathscr{D}^Y} A(r_J) \big(1+O(y_0Y^{-1})\big)
    \\
    &+
    O\left(
    \sum_{\lambda_J\in \mathscr{O}^Y}
    \frac{|A(r_J)|}{1+Y|\vartheta(\lambda_J)|}
    \right).
\end{split}
\end{equation}
We bound the off-diagonal first.
Let $\mathscr{O}$ denote the sum on the second line of \eqref{ch0304:eq052}.
If we can show that the sum is bounded for every $Y\geq 1$, then by Lebesgue dominated convergence
we will conclude that the sum vanishes as $Y\to\infty$.
Ignoring the $\lambda_j<1$ and grouping
the terms corresponding to $a_i\leq \lambda_j\leq a_i+1$, $a_i\in\ZZ^+$, because of \eqref{ch0301:eq004} we have
\[
\mathscr{O}
\ll
\sum_{s=0}^n \sum_{k=1}^\infty \sum_{l=1}^\infty \frac{B_s[k]B_{n-s}[l]}{1+|k-l+O(n)|}
\quad
\text{where}
\quad
B_r[m]=\sum_{a_1+\cdots+a_r=m} (a_1\cdots a_r)^{-\beta}.
\]
It can be proved that $B_r[m]\ll_r m^{(1-\beta)r-1}$, for instance by induction on $r$,
\[B_r[m]\ll \sum_{a_r\leq m} B_{r-1}[m] \ll m^{1-\beta} m^{(1-\beta)(r-1)-1} = m^{(1-\beta)r-1}.\]
We have $1+|k-l+O(n)|\gg_n|k-l|$, so
\begin{equation}\label{1707:eq002}
\mathscr{O} \ll \sum_{k\neq l} \frac{1}{k^ul^v|k-l|} \quad \text{with}\quad
\begin{array}{l}
u = (\beta-1)s+1\\
v = (\beta-1)(n-s)+1.
\end{array}
\end{equation}
Since $\beta>1-1/n$, splitting the summation for $1\leq k\leq l/2$ and $l/2<k<l$
(and similarly for $l$ when $l<k$)
we can upper bound \eqref{1707:eq002} by convergent harmonic series $\sum_{k\geq 1} k^{-1-\eps}$
so that we conclude that $\mathscr{O}$ is bounded for every $Y\geq 1$.

Let $\mathscr{D}$ denote the sum appearing in the first line in \eqref{ch0304:eq052}.
We show that this sum is absolutely convergent for every $Y\geq 1$. Indeed, we can bound
\[
|\mathscr{D}|
\ll
\sum_{s=0}^n \sum_{k=1}^\infty B_s[k]B_{n-s}[l]
\]
with $|l-k|\leq n$, and using $B_r[m]\ll m^{(1-\beta)r-1}$ we obtain
\[
|\mathscr{D}| \ll \sum_{k=1}^\infty \frac{1}{k^{u+v}},
\]
which converges since $u+v>1$. For $Y\to\infty$, we see that the sum in \eqref{ch0304:eq051} is absolutely convergent.
\end{proof}

\begin{proof}[Proof of Theorem \ref{ch0304:thm01}.]
Let us prove that for every $1\leq n\leq p$ we have
\begin{equation}\label{ch0304:eq059}
    \lim_{Y\to\infty} \frac{1}{Y} \int_0^Y \phi(y)^n dy
    =
    \lim_{Y\to\infty} \frac{1}{Y} \int_{y_0}^Y S(y,X(Y))^n dy.
\end{equation}
The result will then follow from Proposition \ref{ch0304:prop01}.
Consider the case $n=p$, and assume first that $p$ is even;
from Proposition \ref{ch0304:prop01}
we know that the $p$-th moment of $S$ is finite.
Hence we can write, using H\"older's inequality,
\begin{equation}\label{ch0304:eq060}
\begin{split}
    &\frac{1}{Y}\int_0^Y \phi(y)^p dy
    =
    \frac{1}{Y}\int_{y_0}^Y S(y,X(Y))^p dy
    +
    O\left(
    \frac{1}{Y}\int_0^{y_0} |\phi(y)|^p dy
    \right)
    \\
    +&O\left(
    \sum_{n=0}^{p-1}\frac{1}{Y}
    \left(\int_{y_0}^Y|S(y,X(Y))|^p dy\right)^{\frac{n}{p}}
    \left(\int_{y_0}^Y|\mathcal{E}(y,X(Y))|^p dy\right)^{\frac{p-n}{p}}
    \right).
\end{split}
\end{equation}
Because of \eqref{ch0301:eq003}, the errors in \eqref{ch0304:eq060}
tend to zero as $Y\to\infty$, and we obtain \eqref{ch0304:eq059}.
Assume now that $p$ is odd. Then we can write
\begin{equation}\label{ch0304:eq061}
\begin{split}
    \frac{1}{Y}\int_0^Y \phi(y)^p dy
    &=
    \frac{1}{Y}\int_{y_0}^Y S(y,X(Y))^p dy
    +
    O\left(
    \frac{1}{Y}\int_0^{y_0} |\phi(y)|^p dy
    \right)
    \\
    &+
    O\left(
    \sum_{n=0}^{p-1} \frac{1}{Y}
    \int_{y_0}^Y |S(y,X(Y))|^n |\mathcal{E}(y,X(Y))|^{p-n}dy
    \right).
\end{split}
\end{equation}
If $p=1$ then the error contains only one term, and
this tends to zero as $Y\to\infty$, in view of \eqref{ch0301:eq003},
and we obtain \eqref{ch0304:eq059}.
For $p>1$, we cannot use H\"older inequality exactly in the same way as in \eqref{ch0304:eq061}
because we cannot bound by $O(1)$ the $p$-th moment of $|S|$ (we can do this without absolute value),
and we argue a little different, by bounding $\mathcal{E}$ by its absolute value.
Proposition \ref{ch0304:prop01} gives 
$Y^{-1}\int_{y_0}^Y|S(y,X(Y))|^n dy\ll 1$ for every $1\leq n\leq p-1$,
and condition \eqref{ap:eq005} shows that in the limit as $Y\to\infty$
the parenthesis tends to zero, thus we obtain \eqref{ch0304:eq059}.
For $n\leq p-1$ the same argument as in \eqref{ch0304:eq060}--\eqref{ch0304:eq061} works.
This proves the theorem.
\end{proof}


\section{Limiting distribution}\label{apf03:section03}

In this section we prove Theorem \ref{ch0303:thm01}.
We start by recalling the definition of limiting distribution.
\begin{definition}
A limiting distribution for a function $\phi:[0,\infty)\to\RR$ is a
probability measure $\mu$ on $\RR$ such that the limit
\begin{equation}
   \lim_{Y\to\infty} \frac{1}{Y} \int_0^Y g(\phi(y))dy = \int_\RR g d\mu
\end{equation}
holds for every bounded continuous function $g$ on $\RR$.
\end{definition}

\noindent
We give now two preparatory lemmata.
Let $X>T>2$ and consider the functions
\begin{equation}\label{ch0303:eq020}
    \phi_T(y)=2\Re(\sum_{0<\lambda_n\leq T} r_ne^{i\lambda_n y}),
    \quad
    \psi_T(y,X)=2\Re(\sum_{T<\lambda_n\leq X} r_ne^{i\lambda_n y})+ \mathcal{E}(y,X).\nonumber
\end{equation}

\begin{lemma}
Assume the hypothesis of Theorem \ref{ch0303:thm01}. Then for $Y,T\gg 1$ with $X(Y)\geq T$,
we have
\begin{equation}\label{ch0303:eq021}
    \frac{1}{Y}\int_{y_0}^Y |\psi_T(y,X(Y))|^2dy
    \ll
    \frac{1}{T^{2\beta-1}}.\nonumber
\end{equation}
\end{lemma}

\begin{proof}
This follows from \cite[Lemma 7.1]{iwaniec_analytic_2004} and \eqref{ch0301:eq004}.
\end{proof}

\begin{lemma}\label{ch0303:lemma00sbo}
For each $T\geq 2$ there exists a probability measure $\nu_T$ on $\RR$
such that
\begin{equation}\label{ch0303:eq024}
    \nu_T(f):=\int_\RR f(x)d\nu_T(x)=\lim_{Y\to\infty}\frac{1}{Y}\int_{y_0}^Y f(\phi_T(y))dy\nonumber
\end{equation}
for every bounded Lipschitz continuous function $f$ on $\RR$.
In addition, there is a constant $c>0$ such that
the support of $\nu_T$ lies in the ball $B(0,cT^{1-\beta})$
for $\beta<1$, respectively $B(0,c\log T)$ for $\beta=1$.
For $\beta>1$ the support of $\nu_T$ is bounded independently of $T$.
\end{lemma}

\begin{proof}
The existence of the measure is \cite[Lemma 2.3]{rubinstein_chebyshevs_1994}.
The statement about the support of $\nu_T$ follows from the fact that
\begin{equation}\label{ch0303:eq025}
    |\phi_T(y)|\ll T^{1-\beta}, \quad
    |\phi_T(y)|\ll \log T, \quad\nonumber
\end{equation}
respectively for $\beta<1$, $\beta=1$, and $|\phi_T(y)|\ll 1$ independently of $T$ for $\beta>1$.
\end{proof}

\begin{proof}[Proof of Theorem \ref{ch0303:thm01}.]
We follow closely \cite[179-181]{rubinstein_chebyshevs_1994}.
Consider a bounded Lipschitz continuous function $f$,
with Lipschitz constant $c_f$, so that
\begin{equation}\label{ch0303:eq026}
    |f(x)-f(y)|\leq c_f|x-y|.
\end{equation}
Then we have for $Y\gg 1$
\begin{equation}\label{ch0303:eq027}
    \frac{1}{Y}\int_{y_0}^Y \! \big(f(\phi(y)) \!-\! f(\phi_T(y))\big)dy
    \ll
    \big(\frac{1}{Y}\int_{y_0}^Y \! |\psi_T(y,X(Y))|^2dy\big)^{1/2}
    \!\ll
    \frac{1}{T^{\beta-1/2}},\nonumber
\end{equation}
so that taking the limit as $Y\to\infty$ we obtain
\begin{equation}\label{ch0303:eq028}
\begin{split}
    \nu_T(f)\!-\!O\left(\frac{1}{T^{\beta-1/2}}\right)
    &\!\leq\!
    \liminf_{Y\to\infty}\frac{1}{Y}\int_{y_0}^Y\!\! f(\phi(y))dy
    \\
    &\!\leq\!
    \limsup_{Y\to\infty}\frac{1}{Y}\int_{y_0}^Y\!\! f(\phi(y))dy
    \!\leq\!
    \nu_T(f)\!+\!O\left(\frac{1}{T^{\beta-1/2}}\right).
\end{split}\nonumber
\end{equation}
Since $T$ can be arbitrarily large, we conclude that the $\liminf$ and $\limsup$
coincide, i.e. that
\begin{equation}\label{ch0303:eq029}
    \mu(f):=\lim_{Y\to\infty}\frac{1}{Y}\int_{y_0}^Y f(\phi(y))dy
\end{equation}
exists. Thus there exists a Borel measure $\mu$ on $\RR$ such that
\eqref{ch0303:eq029} holds for all $f$ satisfying \eqref{ch0303:eq026}.
Moreover, for such $f$,
\begin{equation}\label{ch0303:eq030}
    |\mu(f)-\nu_T(f)|\ll \frac{1}{T^{\beta-1/2}}.
\end{equation}
In view of Lemma \ref{ch0303:lemma00sbo} and \eqref{ch0303:eq030} we also have for $\beta<1$
\begin{equation}\label{ch0303:eq031}
    \mu(B_\lambda^c)
    =\nu_T(B_\lambda^c)+O\left(\frac{1}{T^{\beta-1/2}}\right)
    =O\left(\frac{1}{T^{\beta-1/2}}\right)
\end{equation}
for $\lambda=cT^{1-\beta}$
($B_\lambda^c$ is the complement of the open ball of radius $\lambda$).
This leads to
\begin{equation}\label{ch0303:eq032}
    \mu(B_\lambda^c)=O(\lambda^{-(2\beta-1)/(2-2\beta)}).\nonumber
\end{equation}
In the case when $\beta=1$ we insert $\lambda=c\log T$ in \eqref{ch0303:eq031},
which gives $\mu(B_\lambda^c)=O(e^{-\lambda/2c})$.
Finally, for $\beta>1$, the compactness of the support of $\mu$ follows from the fact
that $\phi_T$ is bounded independently of $T$.
\end{proof}

\begin{proof}[Proof of Corollary \ref{apf01:cor01}]
The proof follows the lines of \cite[Lemma 2.5]{fiorilli_elliptic_2014}.
We show that \eqref{ch0304:eq069} holds for $n=p$,
as the case for $n<p$ is similar.
First observe that by Theorem \ref{ch0304:thm01}
we have the bound as $Y\to\infty$
\begin{equation}\label{ch0304:eq077}
    \frac{1}{Y}\int_0^Y |\phi(y)|^p dy \ll 1.
\end{equation}
Consider for $S\gg 1$ the Lipschitz bounded continuous function
\begin{equation}\label{ch0304:eq070}
H_S(x):=
\begin{cases}
0     & \text{if } |x|\leq S,\\
|x|-S & \text{if } S< |x|\leq S+1,\\
1     & \text{if } |x|> S+1.
\end{cases}\nonumber
\end{equation}
By Theorem \ref{ch0303:thm01} we have
\begin{equation}\label{ch0304:eq071}
\lim_{Y\to\infty}\frac{1}{Y}\int_0^Y H_S(\phi(y))dy
=
\int_\RR H_S(x)d\mu
\ll
S^{-(2\beta-1)/(2-2\beta)}.\nonumber
\end{equation}
It follows that
\begin{equation}\label{ch0304:eq072}
\limsup_{Y\to\infty}\frac{1}{Y}\int_{\genfrac{}{}{0pt}{}{0\leq y\leq Y}{|\phi(y)|\geq S+1}} \!\!\!\!\!\! dR
\leq
\limsup_{Y\to\infty}\frac{1}{Y}\int_0^Y \!\!\! H_S(\phi(y))dy
\ll
S^{-(2\beta-1)/(2-2\beta)}.\nonumber
\end{equation}
In view of the bound \eqref{ch0304:eq077} we can write
\begin{equation}\label{ch0304:eq073}
\begin{split}
    \limsup_{Y\to\infty}\frac{1}{Y}\int_{\genfrac{}{}{0pt}{}{0\leq y\leq Y}{|\phi(y)|\geq S}} |\phi(y)|^p dy
    &=
    \limsup_{Y\to\infty}\sum_{\ell=0}^\infty \frac{1}{Y}\int_{\genfrac{}{}{0pt}{}{0\leq y\leq Y}{|\phi(y)|\geq S+\ell}} |\phi(y)|^p dy
    \\
    &\ll
    \sum_{\ell=0}^\infty (S+\ell+1)^p (S+\ell)^{-(2\beta-1)/(2-2\beta)}
    \\
    &\ll
    S^{p+1-(2\beta-1)/(2-2\beta)}.
\end{split}\nonumber
\end{equation}
Here we used \eqref{ch0304:eq077} to interchange the $\limsup$ with the infinite series,
and the fact that $\beta>1-1/(2p+4)$ to estimate the sum of the series.
Define now the bounded Lipschitz continuous function
\begin{equation}\label{ch0304:eq074}
G_S(x):=
\begin{cases}
x^p  & 0\leq x\leq S,\\
S^p(S+1-x) & S<x\leq S+1,\\
0  & x>S+1,
\end{cases}\nonumber
\end{equation}
for $x\geq 0$, and $G_S(-x)=(-1)^p G_S(x)$ for $x<0$.
We obtain
\begin{equation}\label{ch0304:eq075}
\begin{split}
\limsup_{Y\to\infty}\frac{1}{Y}\int_0^Y \phi(y)^pdy
&=
\limsup_{Y\to\infty}\frac{1}{Y}\int_0^Y G_S(\phi(y))dy
\\
&\phantom{xxxxxxxxxx}+ O(S^{p+1-(2\beta-1)/(2-2\beta)})
\\
&=
\int_\RR G_S(x)d\mu + O(S^{p+1-(2\beta-1)/(2-2\beta)})
\\
&=
\int_\RR x^p d\mu + O(S^{p+1-(2\beta-1)/(2-2\beta)}).
\end{split}
\end{equation}
Taking $S\to\infty$ we conclude that
\begin{equation}\label{ch0304:eq076}
    \limsup_{Y\to\infty}\frac{1}{Y}\int_0^Y \phi(y)^pdy=\int_\RR x^p \, d\mu.\nonumber
\end{equation}
A similar argument works for the liminf, and this concludes the proof.
\end{proof}


\section{Applications to hyperbolic counting}\label{apf_04_section04}

In this section we prove Theorem \ref{ch00:intro:thm01}.
The strategy of proof consists
in applying the pretrace formula \cite[Th. 7.4]{iwaniec_spectral_2002}
to a regularized version of the automorphic kernel
\begin{equation}\label{ch02cs01:eq001}
    N(s,z,w)
    = \sum_{\gamma\in\Gamma} \mathbf{1}_{[0,s]}(d(z,\gamma w))
    = \sum_{\gamma\in\Gamma} \mathbf{1}_{[0,(\cosh s-1)/2]}(u(z,\gamma w)).
\end{equation}
We construct here the smoothing and explain how the \enquote{complete}
main term $M(s,z,w)$ is defined both for the sharp and for the regularized problem
(see \eqref{ch02:def:mainterm} and \eqref{def:Mpm}).

Let $\delta$ be a small positive real number, $0<\delta<1$, and consider the function
\begin{equation}\label{ch02:def:k_delta}
k_\delta(u):=\frac{1}{4\pi\sinh^2(\delta/2)}\mathbf{1}_{[0,(\cosh(\delta)-1)/2]}(u)\nonumber
\end{equation}
where $\mathbf{1}_{[0,A]}$ is the indicator function of the set $[0,A]$.
For $u=u(z,w)$, the function $k_\delta$ is the indicator function of a ball of radius $\delta$
in the hyperbolic plane, normalized so that the ball has unit volume.
In other words, it satisfies
\begin{equation}\label{k_delta_integrates_to_one}
\int_\HH k_\delta(u(z,w))d\mu(z) = 1\nonumber.
\end{equation}
Define $k^\pm(u)$ as the functions given by the convolution product
\begin{equation}\label{def:kpm}
k^\pm(u):=\left(\mathbf{1}_{[0,(\cosh(s\pm\delta)-1)/2]}\,\ast\,k_\delta\right)(u),
\end{equation}
where the hyperbolic convolution of two functions $k_1,k_2$
is defined \cite[(2.11)]{chamizo_applications_1996} as
\begin{equation}\label{1207:eq002}
    \int_\HH k_1(u(z,v))k_2(u(v,w))d\mu(v).
\end{equation}
Notice that the convolution of $k_1$ and $k_2$ in \eqref{1207:eq002}
only depends on the distance between $z$ and $w$, and thus defines
a function of $u=u(z,w)$, as written in compact form in \eqref{def:kpm}.
Because of the triangle inequality $d(z,w)\leq d(z,v)+d(v,w)$,
for $Z\geq 0$ the convolution $\mathbf{1}_{[0,\cosh(Z)-1)/2]}\ast k_\delta$ satisfies
\begin{equation}\label{convolution_values}
(\mathbf{1}_{[0,\cosh(Z)-1)/2]}\ast k_\delta) (u(z,w))
=
\begin{cases}
1 & d(z,w)\leq Z-\delta\\
0 & d(z,w)\geq Z+\delta.
\end{cases}\nonumber
\end{equation}
From this we deduce that
\begin{equation}\label{kpm_ineq}
k^-(u) \leq \mathbf{1}_{[0,(\cosh(s)-1)/2]}(u) \leq k^+(u),\nonumber
\end{equation}
and summing over $\gamma\in\Gamma$:
\begin{equation}\label{Kpm_ineq}
K^-(s,\delta)=\sum_{\gamma\in\Gamma}k^-(u(z,\gamma w))
\leq
N(s)
\leq
\sum_{\gamma\in\Gamma}k^+(u(z,\gamma w))=K^+(s,\delta).
\end{equation}

We want to expand $K^\pm(s,\delta)$ using the pretrace formula.
In order to do this we need to
prove that Selberg--Harish-Chandra transform of $k^\pm$ is an admissible test function.

The Selberg--Harish-Chandra (SHC) transform turns convolutions
into products (see \cite[p. 323]{chamizo_applications_1996}),
so if we denote by $h_s$ the SHC transform of $\mathbf{1}_{[0,(\cosh(s)-1)/2]}$,
and by $h^\pm$ the SHC transform of $k^\pm$, then we have
\begin{equation}\label{h_of_conv}
h^\pm (t) = \frac{1}{4\pi\sinh^2(\delta/2)} h_{s\pm\delta}(t) h_\delta(t).
\end{equation}
Denote for simplicity
\begin{equation}\label{h_tilde}
\tilde{h}_\delta(t)=\frac{1}{4\pi\sinh^2(\delta/2)}h_\delta(t).\nonumber
\end{equation}
The function $h_s(t)$ is explicitely computed in \cite[eq. (2.6)]{chamizo_applications_1996}
and \cite[eq. (2.10)]{phillips_circle_1994}, and is given by
\begin{equation}
    h_s(t)=2^{3/2}\int_{-s}^s (\cosh s-\cosh u)^{1/2} e^{itu}du.\nonumber
\end{equation}
Observe that $h_s(t)$ is a holomorphic function of $t$.
Notice also that for every $t\in\RR$ and $s>0$ we have the estimate
\begin{equation}\label{h_R(0)}
    |h_s(t)|\leq h_s(0)\leq se^{s/2}.
\end{equation}
This will be useful in later estimates for $h(t)$ for $t$ close to $0$.

Lemma 2.4 in \cite{chamizo_applications_1996} shows that
$h_s(t)$ can be expressed in terms of special functions:
for every $s>0$, and every $t\in\CC$ such that $it\not\in\ZZ$, we have
\begin{equation}\label{ch02:h_R_special_functions}
h_s(t)=
2\sqrt{2\pi\sinh s}\;
\Re\left(
e^{its} \frac{\Gamma(it)}{\Gamma(3/2+it)} F\left(-\frac{1}{2};\frac{3}{2};1-it;\frac{1}{(1-e^{2s})}\right)
\right),
\end{equation}
where $F$ is the Gauss hypergeometric function.
Looking at the series expansion of~$F$ we can write,
for $t\in\RR$ and $s>\frac{1}{2}\log 2$,
\begin{equation}\label{ch02:F_expansion}
F\left(-\frac{1}{2};\frac{3}{2};1-it;\frac{1}{(1-e^{2s})}\right)
=
1+O\left(e^{-2s}\min\left\{1,\frac{1}{|t|}\right\}\right).
\end{equation}
Inserting \eqref{ch02:F_expansion} in \eqref{ch02:h_R_special_functions},
using the Taylor expansion of the hyperbolic sine,
and Stirling's formula to estimate the quotient of Gamma functions, we obtain
the simpler expression
\begin{equation}\label{0707:eq001}
    h_s(t)
    =
    2\sqrt{\pi} e^{s/2} \Re\left(e^{its}\frac{\Gamma(it)}{\Gamma(3/2+it)}\right)
    +
    O\left(\frac{e^{-3s/2}}{|t|(1+|t|^{1/2})}\right).
\end{equation}
We are also interested in purely imaginary values of $t$ in the interval $[-i/2,i/2]$.
From \cite[Lemma 2.4]{chamizo_applications_1996}
and \cite[Lemma 2.1]{phillips_circle_1994} we have, for $s\geq 1$ and $t$ purely imaginary,
\begin{equation}\label{ch02:hR_t_imag}
h_s(t)=
\sqrt{2\pi\sinh s} e^{s|t|} \frac{\Gamma(|t|)}{\Gamma(3/2+|t|)}
+O\left(\big(1+|t|^{-1}\big) e^{s\left(\frac{1}{2}-|t|\right)}\right).
\end{equation}
For $0\leq s\leq 1$ and $t\in\CC$ we can write instead (see \cite[Lemma 2.4 (c)]{chamizo_applications_1996})
\begin{equation}\label{ch02:smallR}
h_s(t)=
2\pi s^2 \frac{J_1(st)}{st} \sqrt{\frac{\sinh s}{s}} + O\left(s^2 e^{s|\Im t|}\min\{s^2,|t|^{-2}\}\right),
\end{equation}
where $J_1(z)$ is the $J$-Bessel function of order $1$.


\subsection{Contribution from the small eigenvalues}\label{ch02:subsubsection03}

In the pretrace formula we split the spectral expansion
into the contribution associated to the small eigenvalues
and that associated to the rest of the spectrum.
We compute now, using the expressions given above for $h_s(t)$,
the contribution of the small eigenvalues
(we also include the case $\lambda=1/4$).
We discuss first the contribution coming from the discrete spectrum,
and then the contribution coming from the continuous spectrum at $\lambda=1/4$.

For the eigenvalue $\lambda_0=0$, i.e. $t_0=i/2$, we have a simple formula for $h_s(i/2)$, namely
\begin{equation}\label{lambda0}
h_s(i/2)
= 2\pi(\cosh s -1).\nonumber
\end{equation}
This means that we can compute $h^\pm(i/2)$ directly and obtain
\begin{equation}\label{hpm_iover2}
h^\pm(i/2)
=
\frac{1}{4\pi\sinh^2(\delta/2)} h_{s\pm\delta}(i/2) h_\delta(i/2)
=
2\pi(\cosh s -1) +O(\delta e^s).
\end{equation}
For $\lambda=1/4$ we have again a simple expression for $h_s(0)$
(see \cite[Lemma 2.2]{phillips_circle_1994}):
\begin{equation}\label{hR0}
h_s(0)=4\big(s+2(\log 2-1)\big)e^{s/2}+O(e^{-s/2}).
\end{equation}
Using this and \eqref{ch02:smallR} we can compute the value $h^\pm(0)$.
First observe that the function $J_1(z)$ verifies
\begin{equation}\label{limJ1}
\lim_{z\to 0} \frac{J_1(z)}{z}=\frac{1}{2}.
\end{equation}
Using the Taylor expansion of the hyperbolic sine for $\delta\ll 1$, we get,
in the limit as $t\to 0$ in \eqref{ch02:smallR},
\begin{equation}\label{htilde0}
\tilde{h}_\delta(0)
=
\frac{1}{4\pi\sinh^2(\delta/2)} h_\delta(0)
=
1+O(\delta^2).
\end{equation}
Combining \eqref{hR0}, \eqref{limJ1}, and \eqref{htilde0}, we obtain
\begin{equation}\label{hpm0}
h^\pm(0)
\!=\!
h_{s\pm\delta}(0) \tilde{h}_\delta(0)
\!=\!
4\big(s+2(\log 2-1)\big)e^{s/2} \!+\! O(s\,\delta\,e^{s/2} + e^{-s/2}).
\end{equation}
We analyze now the contribution coming from the small eigenvalues $0<\lambda_j<1/4$.
These eigenvalues correspond to $t_j$ chosen so that $t_j/i\in(0,1/2)$. It is important
to recall that there is only a finite number of such eigenvalues, which implies
that there exists $0<\eps_\Gamma\leq 1/4$ such that $t_j/i\in(\eps_\Gamma,1/2-\eps_\Gamma)$.
For our analysis we make use of equations \eqref{ch02:hR_t_imag} and \eqref{ch02:smallR}. 
We can write for $t=t_j$ purely imaginary corresponding to a small eigenvalue:
\begin{align}
h^\pm(t)
&=
\left(\!
\sqrt{2\pi\sinh(s\pm\delta)} e^{(s\pm\delta)|t|} \frac{\Gamma(|t|)}{\Gamma(3/2+|t|)}
\!+\!
O\left(\!(1\!+\!|t|^{-1})e^{(1/2-|t|)(s\pm\delta)}\!\right)
\!\right)
\nonumber
\\
&\times
\left(
\frac{2\pi\delta^2}{4\pi\sinh^2(\delta/2)} \frac{J_1(\delta t)}{\delta t} \sqrt{\frac{\sinh\delta}{\delta}}
+
O\left(\frac{\delta^2e^{\delta|t|}\min\{\delta^2,|t|^{-2}\}}{\sinh^2(\delta/2)}\right)
\right).\label{hpm_small}
\end{align}
Using Taylor approximations in the above formula, for $\delta\ll 1$ and $t=t_j$,
we can rewrite \eqref{hpm_small} in the more comfortable way
\begin{equation}\label{hpm_small2}
\begin{gathered}
h^\pm(t)
=
\bigg[
\frac{\Gamma(|t|)}{\Gamma(3/2+|t|)}
\left( \sqrt{\pi}e^{s/2} + O(\delta e^{s/2} + e^{-3s/2}) \right)
\left( e^{s|t|} + O(\delta e^{s|t|} \right)
\\
+
O\left(e^{s(1/2-\eps_\Gamma)}\right)
\bigg]
\bigg[
1 + O(\delta + \delta^2 e^{\delta|t|})
\bigg].
\end{gathered}\nonumber
\end{equation}
Expanding the product we obtain that the contribution from a given small eigenvalue
$0<\lambda_j<1/4$ is given by
\begin{equation}\label{hpm_small_final}
h^\pm(t_j)
=
\sqrt{\pi} \frac{\Gamma(|t_j|)}{\Gamma(3/2+|t_j|)} e^{s(1/2+|t_j|)}
+
O(\delta e^{s(1-\eps_\Gamma)} + e^{s(1/2-\eps_\Gamma)}).
\end{equation}
Finally we discuss the contribution coming from the Eisenstein series at $\lambda=1/4$.
By this contribution we mean the expression
\begin{equation}
\frac{1}{4\pi} \sum_{\mathfrak{a}} E_\mathfrak{a}(z,1/2)\overline{E_\mathfrak{a}(w,1/2}) \; \int_{-\infty}^{+\infty} h^\pm(t)  \; dt,\nonumber
\end{equation}
where the sum runs over the cusps of $\Gamma\backslash\HH$.
Consider now the integral
\begin{equation}\label{continuous1}
\begin{aligned}
\int_{-\infty}^{+\infty} h^\pm(t)dt
=
\int_{-\infty}^{+\infty} h_{s\pm\delta}(t)dt
+
\int_{-\infty}^{+\infty} O(|h_{s\pm\delta}(t)(\tilde{h}_\delta(t)-1)|)dt.
\end{aligned}
\end{equation}
Using \eqref{ch02:smallR} and \eqref{limJ1} we get, for $\delta\ll 1$ and $t\in\RR$,
\begin{equation}\label{ch02:htilde_estimate}
\tilde{h}_\delta(t)
=
\begin{cases}
1 + O(\delta|t|+\delta^2) & \delta|t|<1\\
O\left(\frac{1}{(\delta|t|)^{3/2}}\right) & \delta|t|\geq 1.
\end{cases}
\end{equation}
By \eqref{h_R(0)}, \eqref{ch02:h_R_special_functions}, and \eqref{ch02:F_expansion},
we have instead
\begin{equation}\label{hRbounds}
h_{s\pm\delta}(t)=O\left( e^{s/2}\;\min\left\{s,\frac{1}{|t|},\frac{1}{|t|^{3/2}}\right\} \right).
\end{equation}
Inserting this in the second integral in \eqref{continuous1},
using \eqref{ch02:htilde_estimate},
and splitting the integral to optimize the above estimate, we obtain for $s\geq 1$
\begin{equation}
\begin{split}
    \int_{-\infty}^{+\infty} \!\!\! |h_{s\pm\delta}(t)(\tilde{h}_\delta(t)-1)|dt
    \ll
    s\,\delta^{1/2}e^{s/2}\rule{0pt}{15pt}.\nonumber
\end{split}
\end{equation}
Hence we can write
\begin{equation}\label{continuous2}
\int_{-\infty}^{+\infty} h^\pm(t)dt
=
\int_{-\infty}^{+\infty} h_{s\pm\delta}(t)dt
+
O(s\delta^{1/2} e^{s/2}).\nonumber
\end{equation}
The average of the function $h_s(t)$ can be computed via the Fourier inversion theorem, giving
\begin{equation}
\int_\RR h_s(t)dt
=
2\pi g_s(0),\nonumber
\end{equation}
where $g_s(u)$ is the Fourier inverse of $h_s(t)$ and is given by (see \cite[eq. (2.9)]{phillips_circle_1994})
\begin{equation}\label{g_function}
g_s(u)=
\begin{cases}
2^{3/2}(\cosh s-\cosh u)^{1/2} & |u|\leq s\\
0 & \text{otherwise.}
\end{cases}\nonumber
\end{equation}
For $s\pm\delta$ and $u=0$ we obtain
\begin{equation}\label{gR0}
\begin{aligned}
2\pi g_{s\pm\delta}(0)
= 2^{5/2}\pi (\cosh(s\pm\delta)-1)^{1/2}
= 4\pi e^{s/2} + O(\delta e^{s/2} + e^{-s/2}).
\end{aligned}\nonumber
\end{equation}
This shows that we can write
\begin{equation}\label{hpm_continuous_final}
\int_\RR h^\pm(t)dt
=
4\pi e^{s/2} + O(s\delta^{1/2} e^{s/2} + e^{-s/2}).
\end{equation}


\subsubsection{Definition of the main term.}
We define the complete main term associated to the hyperbolic circle problem as
\begin{equation}\label{ch02:def:mainterm}
\begin{gathered}
M(s,z,w)
:=
\frac{\pi e^s}{\mathrm{vol}(\Gamma\backslash\HH)}
+
\sqrt{\pi} \sum_{t_j\in(0,\frac{i}{2})} \frac{\Gamma(|t_j|)}{\Gamma(3/2+|t_j|)} e^{s(1/2+|t_j|)} \phi_j(z)\overline{\phi_j(w})
\\
+
4\big(s+2(\log 2-1)\big)\,e^{s/2}\,\sum_{t_j=0} \phi_j(z)\overline{\phi_j(w})
+
e^{s/2}\,\sum_{\mathfrak{a}} E_\mathfrak{a}(z,1/2)\overline{E_\mathfrak{a}(w,1/2}).
\end{gathered}
\end{equation}
Denote by $M^\pm(s,\delta)$ the main term associated to $K^\pm(s,\delta)$,
defined as the contribution coming from the eigenvalues $\lambda_j\leq 1/4$
(together with the contribution from the continuous spectrum at $\lambda=1/4$)
for the kernels $K^\pm$. In other words,
\begin{equation}\label{def:Mpm}
\begin{aligned}
M^\pm(s,\delta)
&:=
\sum_{t_j\in[0,\frac{i}{2}]} h^\pm(t_j)\phi_j(z)\overline{\phi_j(w})
\\
&+
\frac{1}{4\pi}\sum_\mathfrak{a} E_\mathfrak{a}(z,1/2)\overline{E_\mathfrak{a}(w,1/2}) \int_\RR h^\pm(t)dt.
\end{aligned}
\end{equation}
Then we can summarize equations \eqref{hpm_iover2}, \eqref{hpm0}, \eqref{hpm_small_final}, and \eqref{hpm_continuous_final} by saying that
\begin{equation}\label{Mpm_vs_M}
M^\pm(s,\delta,z,w)=M(s)+O\left(\delta e^s + s\delta^{1/2} e^{s/2} + e^{s(1/2-\eps_\Gamma)}\right)
\end{equation}
(in order to simplify notation we have omitted the $z,w$ dependence of $M$ and $M^\pm$).


\subsection{Proof of Theorem \ref{ch00:intro:thm01}}
From the inequality \eqref{Kpm_ineq} we have
\begin{equation}\label{cha:NtoKpm}
    |E(s,z,w)|\leq \max\{\,|K^-(s,\delta)-M(s)|,|K^+(s,\delta)-M(s)|\,\}.\nonumber
\end{equation}
Squaring and integrating we see that in order to prove Theorem \ref{ch00:intro:thm01}
it suffices to prove that for some $\delta=\delta(T)$ we have
\begin{equation}\label{cha:kpmtobound}
    \int_T^{T+1} \left| \frac{K^\pm(s,\delta)-M(s)}{e^{s/2}} \right|^2 ds \ll T.
\end{equation}
Using \eqref{Mpm_vs_M} we have
\begin{equation}\label{Kpm_minus_M_vs_Mpm}
    K^\pm(s,\delta)\!-\!M(s)
    =
    K^\pm(s,\delta)\!-\!M^\pm(s)
    \!+\!
    O\left(\delta e^s \!+\! s \delta^{1/2}e^{s/2} \!+\! e^{s(1/2-\eps_\Gamma)}\right).
\end{equation}
We will show later that the choice $\delta(T)=e^{-T/2}$ is admissible.
This leads to
\begin{equation}\label{cha:HF}
    \int_T^{T+1} \! \left|\frac{K^\pm(s,\delta)-M(s)}{e^{s/2}}\right|^2 \! ds
    =
    H(T) + O\big(H(T)^{1/2} + 1\big),
\end{equation}
where
\begin{equation}\label{cha:HFdef}
H(T):=\int_T^{T+1} \left|\frac{K^\pm(s,\delta)-M^\pm(s)}{e^{s/2}}\right|^2 ds.
\end{equation}
The problem reduces thus to show that $H(T)\ll T$, which we can prove using the pretrace formula.
Before proceeding to do so, though, we insert a weight function
in \eqref{cha:HFdef}, which turns out useful for multiple integration by parts.
Let $\psi(s)\in C_c^\infty(\RR)$, with $\mathrm{supp}(\psi)\subseteq [-1/2,3/2]$,
such that $0\leq\psi\leq 1$ and $\psi(s)=1$ for $s\in[0,1]$.
Define, for $T\geq 0$, $\psi_T(s)=\psi(s-T)$. Then we have the inequality
\begin{equation}\label{0707:eq002}
    H(T)\leq \int_\RR \left|\frac{K^\pm(s,\delta)-M^\pm(s)}{e^{s/2}}\right|^2 \, \psi_T(s) \, ds,
\end{equation}
and we give bounds on this last integral.
The pretrace formula applied to the function $K^\pm(s,\delta)$ gives
\begin{equation}\label{real_contribution}
    K^\pm(s,\delta)\!-\!M^\pm(s,\delta)
    \!=\!
    \sum_{t_j>0} h^\pm(t_j) \phi_j(z)\overline{\phi_j(w})
    +
    \frac{1}{4\pi} \sum_\mathfrak{a} \int_{-\infty}^{+\infty}\!\!\! h^\pm(t) E_\mathfrak{a}(t) dt,
\end{equation}
where in order to shorten the notation we have written
\[
E_\mathfrak{a}(t)
=
E_\mathfrak{a}(z,1/2\!+\!it)\overline{E_\mathfrak{a}(w,1/2\!+\!it})\!-\!E_\mathfrak{a}(z,1/2)\overline{E_\mathfrak{a}(w,1/2}).
\]
The series and the integral in \eqref{real_contribution}
are absolutely convergent, as for $\delta>0$ we have
$h^\pm(t)\ll t^{-3}$ (this follows from the definition \eqref{h_of_conv} of $h^\pm(t)$ and
the estimates \eqref{ch02:htilde_estimate} and \eqref{hRbounds}).
We also use, for fixed $z\in\HH$, the following standard inequality (for a proof see \cite[\S 13.2]{iwaniec_spectral_2002}):
\begin{equation}\label{crucial_ineq}
\sum_{T\leq t_j\leq T+1} |\phi_j(z)|^2
+
\frac{1}{4\pi}\sum_\mathfrak{a}\int_T^{T+1} |E_\mathfrak{a}(z,1/2+it)|^2 dt
\ll T.
\end{equation}
In order to give upper bounds on the second moment of $K^\pm(s,\delta)-M^\pm(s,\delta)$,
it suffices to estimate separately the square of the series and the square of the integrals
in \eqref{real_contribution}. We prove a lemma that is useful for this purpose.

\begin{lemma}\label{cha:bounds_on_h}
Let $T\gg 1$ and $0<\delta\ll 1$. Let $t_1,t_2\in\RR$, $t_1,t_2\neq 0$. Then
\begin{equation}\label{cha:lemma:eq}
\begin{aligned}
    \int_\RR \! \frac{h_{s\pm\delta}(t_1)\overline{h_{s\pm\delta}(t_2})}{e^s} \psi_T(s) ds
    \!\ll\!
    \frac{g(t_1)g(t_2)}{1+|t_1-t_2|^2}
    \!+\!
    \frac{g(t_1)g(t_2)}{1+|t_1+t_2|^2}
    \!+\!
    \frac{g(t_1)g(t_2)}{e^{2T}}
\end{aligned}
\end{equation}
where $g(t)=(|t|(1+\sqrt{|t|}))^{-1}$.
\end{lemma}

\begin{proof}
From \eqref{0707:eq001} we can write
\begin{equation}
\begin{aligned}
    \frac{h_{s\pm\delta}(t_1)\overline{h_{s\pm\delta}(t_2})}{e^s}
    &=
    \pi e^{\pm\delta}
    \Re\left(e^{i(s\pm\delta)(t_1-t_2)}G(t_1)\overline{G(t_2)}\right)
    \\
    &+
    \pi e^{\pm\delta}
    \Re\left(e^{i(s\pm\delta)(t_1+t_2)}G(t_1)G(t_2)\right)
    +
    O(g(t_1)g(t_2)e^{-2s}),
\end{aligned}
\end{equation}
where $G(t)=\Gamma(it)/\Gamma(3/2+it)$, and $G(t)\ll g(t)$.
After multiplying by $\psi_T(s)$, the integral over $s$ can be
estimated by taking the absolute value inside the integral, or
integrating by parts and then bounding the result. The two different bounds
lead to \eqref{cha:lemma:eq}.
\end{proof}

\subsubsection{Discrete spectrum}
Consider the series in \eqref{real_contribution}.
Taking absolute value and squaring gives a double sum, which we then
multiply by $\psi_T(s)$ and integrate over~$s$.
For simplicity we assume that $z=w$, but the same argument works for $z\neq w$.
We claim that the following holds:
\begin{equation}\label{cha:prop:discrete_eq}
    \sum_{t_j,t_{\ell}>0}
    |\phi_j(z)|^2|\phi_{\ell}(z)|^2
    \int_\RR \frac{ h^\pm(t_j) \overline{ h^\pm(t_{\ell}}) }{e^s} \psi_T(s) ds
    \ll
    \log(\delta^{-1})+\delta^{-1}e^{-2T}+1.
\end{equation}
By symmetry of the estimate \eqref{cha:lemma:eq} in Lemma \ref{cha:bounds_on_h},
and positivity of the integral in \eqref{cha:prop:discrete_eq} for $t_j=t_\ell$,
it is sufficient to consider only the case when $t_{\ell}\geq t_j$.

We follow here a classical argument to analyse the sum (see e.g. \cite{cramer_mittelwertsatz_1922,cramer_uber_1922}).
We consider the part of the series in \eqref{cha:prop:discrete_eq} where $t_j$ and $t_\ell$
are close to each other.
Using \eqref{ch02:htilde_estimate}, \eqref{crucial_ineq}, and Lemma \ref{cha:bounds_on_h},
we have the following estimates:
\begin{equation}\label{cha:disc1}
\begin{aligned}
    &
    \sum_{\genfrac{}{}{0pt}{}{t_j>0}{t_j\leq t_{\ell}< t_j+1}}
    |\phi_j(z)|^2 |\phi_{\ell}(z)|^2 \int_\RR \frac{ h^\pm(t_j) \overline{ h^\pm(t_{\ell}}) }{e^s} \psi_T(s) ds
    \\
    &\ll
    \sum_{\genfrac{}{}{0pt}{}{t_j<\delta^{-1}}{t_j\leq t_{\ell}< t_j+1}}
    \frac{|\phi_j(z)|^2 |\phi_{\ell}(z)|^2}{(t_j t_{\ell})^{3/2}}
    +\!\!\!\!
    \sum_{\genfrac{}{}{0pt}{}{t_j\geq\delta^{-1}}{t_j\leq t_{\ell}< t_j+1}}
    \frac{|\phi_j(z)|^2 |\phi_{\ell}(z)|^2}{(\delta t_j t_{\ell})^3}
    \\
    &\ll
    \sum_{t_j<\delta^{-1}} \frac{|\phi_j(z)|^2}{t_j^2}
    +
    \frac{1}{\delta^3} \sum_{t_j\geq\delta^{-1}} \frac{|\phi_j(z)|^2}{t_j^5} \rule{0pt}{18pt}
    \ll
    \log\delta^{-1}+1.
\end{aligned}
\end{equation}
This shows that a neighbourhood of the diagonal $t_j=t_{\ell}$ (of width 1)
gives a contribution of the order $O(\log\delta^{-1}+1)$.
Similarly we can estimate
\begin{align}
    \sum_{\genfrac{}{}{0pt}{}{t_j\geq\delta^{-1}}{t_{\ell}\geq t_j+1}}
    &
    |\phi_j(z)|^2|\phi_{\ell}(z)|^2 \int_\RR \frac{ h^\pm(t_j) \overline{ h^\pm(t_{\ell}}) }{e^s} \psi_T(s) ds
    \nonumber\\
    &\ll
    \frac{1}{\delta^3} \sum_{t_j\geq\delta^{-1}}
    \frac{|\phi_j(z)|^2}{t_j^3}  \sum_{t_{\ell}\geq t_j+1} 
    \frac{|\phi_{\ell}(z)|^2}{t_{\ell}^3}\left(\frac{1}{|t_{\ell}-t_j|^2}+e^{-2T}\right),
    \nonumber
\end{align}
and using now a unit intervals decomposition we obtain the bound
\begin{align}
    &\quad\ll
    \frac{1}{\delta^3} \sum_{t_j\geq\delta^{-1}} \frac{|\phi_j(z)|^2}{t_j^3} 
    \left(
    \sum_{n=1}^\infty \sum_{n\leq t_\ell-t_j\leq n+1}
    \frac{|\phi_{\ell}(z)|^2}{t_{\ell}^3\,|t_{\ell}-t_j|^2}
    +
    \frac{e^{-2T}}{t_j}
    \right)
    \nonumber\\
    &\quad\ll
    \frac{1}{\delta^3} \sum_{t_j\geq\delta^{-1}} \frac{|\phi_j(z)|^2}{t_j^3} 
    \left(
    \sum_{n=1}^\infty \frac{1}{n^2(t_j+n)^2}
    +
    \frac{e^{-2T}}{t_j}
    \right)
    \label{cha:disc2}\\
    &\quad\ll
    \frac{1}{\delta^3} \sum_{t_j\geq\delta^{-1}} \frac{|\phi_j(z)|^2}{t_j^5}  + \frac{|\phi_j(z)|^2}{t_j^4}e^{-2T}
    \ll
    1+\delta^{-1}e^{-2T}.\nonumber
\end{align}
This shows that the \enquote{tail} of the double series in \eqref{cha:prop:discrete_eq}
gives a contribution of the order $O(1+\delta^{-1}e^{-2T})$.
Finally, the sum for $t_j<\delta^{-1}$ and $t_\ell$ large can be bounded as follows:
\begin{equation}\label{cha:disc3}
\begin{aligned}
    \sum_{\genfrac{}{}{0pt}{}{t_j<\delta^{-1}}{t_{\ell}\geq t_j+1}}
    &
    |\phi_j(z)|^2|\phi_{\ell}(z)|^2 \int_\RR \frac{ h^\pm(t_j) \overline{ h^\pm(t_{\ell}}) }{e^s} \psi_T(s) ds
    \\
    &
    \ll
    \sum_{t_j<\delta^{-1}} \frac{|\phi_j(z)|^2}{t_j^{3/2}}
    \left(
    \sum_{t_{\ell}\geq t_j+1} \frac{|\phi_{\ell}(z)|^2}{t_{\ell}^3/2\,|t_{\ell}-t_j|^2}
    +\right.
    \\
    &
    \qquad+
    \left.
    \sum_{t_j+1\leq t_{\ell}\leq \delta^{-1}} \frac{|\phi_{\ell}(z)|^2}{t_{\ell}^{3/2}} e^{-2T}
    +
    \frac{1}{\delta^{3/2}}
    \sum_{t_{\ell}\geq \delta^{-1}} \frac{|\phi_{\ell}(z)|^2}{t_{\ell}^3} e^{-2T}
    \right),
\end{aligned}\nonumber
\end{equation}
and again a unit intervals decomposition gives
\begin{equation}
\begin{aligned}
    &\ll
    \sum_{t_j<\delta^{-1}} \frac{|\phi_j(z)|^2}{t_j^{3/2}}
    \left(
    \sum_{n=1}^{\infty} \frac{1}{n^2(t_j+n)^{1/2}}
    +
    \delta^{-1/2}e^{-2T}
    \right)
    \\
    &\ll
    \sum_{t_j<\delta^{-1}} \frac{|\phi_j(z)|^2}{t_j^2}
    +
    \delta^{-1}e^{-2T}
    \ll
    \log \delta^{-1} + \delta^{-1}e^{-2T}.
\end{aligned}
\end{equation}
Adding together \eqref{cha:disc1}, \eqref{cha:disc2}, and \eqref{cha:disc3},
we conclude that \eqref{cha:prop:discrete_eq} holds, as claimed.

\subsubsection{Continuous spectrum}
Consider, for a fixed cusp $\mathfrak{a}$, the associated integral in \eqref{real_contribution}.
As in the case of the discrete spectrum, we take absolute value,
square, multiply by $\psi_T(s)$, and then integrate over $s$.
Again we assume that $z=w$, but the same argument works for $z\neq w$.
We claim that the following holds:
\begin{equation}\label{cha:anticipation2}
    \int_\RR \; \Big| \int_{-\infty}^{+\infty} \frac{h^\pm(t)}{e^{s/2}} E_\mathfrak{a}(t) dt \,\Big|^2 \psi_T(s) ds
    \ll
    \log\delta^{-1} + 1 +\delta^{-1}e^{-2T}.
\end{equation}
Since both $h^\pm(t)$ and $E_\mathfrak{a}(t)$ are even in $t$,
we can restrict the domain of integration
(for the inner integral) to the positive real axis $[0,+\infty)$.
So we need to prove
\begin{equation}\label{cha:cont_start2}
    \int_\RR \int_0^\infty \!\!\! \int_0^\infty \!\!\! E_\mathfrak{a}(t_1)E_\mathfrak{a}(t_2)
    \frac{h^\pm(t_1)\overline{h^\pm(t_2})}{e^s} \psi_T(s) dt_1 dt_2 ds
    \ll
    \log\delta^{-1} \!+ 1 + \delta^{-1}e^{-2T}\!.
\end{equation}
Since the integrals are absolutely convergent, we can interchange the order of integration.
Moreover, by symmetry and positivity of the integral for $t_1=t_2$,
it suffices to discuss the integral over $t_2\geq t_1$.

The analysis is similar to the case of the discrete spectrum,
except for $t_1,t_2$ close to zero, when we exploit the
fact that the Eisenstein series is regular at $1/2$ to bound $E_\mathfrak{a}(t)=O(t)$.
Using this together with \eqref{ch02:htilde_estimate}, \eqref{crucial_ineq}, and Lemma \ref{cha:bounds_on_h},
we can bound a unit neighbourhood of the diagonal $t_1=t_2$ by
\begin{equation}\label{0910:eq001}
    \int_0^\infty \int_{t_1}^{t_1+1} E_\mathfrak{a}(t_1) E_\mathfrak{a}(t_2)
    \int_\RR \frac{h^\pm(t_1)\overline{h^\pm(t_2})}{e^s} \psi_T(s) ds \, dt_2 \, dt_1
    \ll
    \log(\delta^{-1}) + 1.
\end{equation}
The tail of the integral can be bounded by
\begin{equation}\label{0910:eq002}
    \int_{\delta^{-1}}^{+\infty} \!\! \int_{t_1+1}^{+\infty} E_\mathfrak{a}(t_1) E_\mathfrak{a}(t_2) \!
    \int_\RR \frac{h^\pm(t_1)\overline{h^\pm(t_2})}{e^s} \psi_T(s) ds \, dt_2 \, dt_1
    \ll
    \delta^{-1}e^{-2T}+1.
\end{equation}
Finally, the range with $t_1<\delta^{-1}$ and $t_2$ large is bounded similarly by
\begin{equation}\label{0910:eq003}
\ll \log(\delta^{-1})+\delta^{-1}e^{-2T}+1.
\end{equation}
Summing \eqref{0910:eq001}, \eqref{0910:eq002}, and \eqref{0910:eq003},
we conclude that \eqref{cha:anticipation2} holds.
Inserting \eqref{cha:prop:discrete_eq} and \eqref{cha:anticipation2} into \eqref{0707:eq002},
and choosing $\delta=e^{-T/2}$, we see that we have proved the bound
\[ H(T)\ll T, \]
as we wanted. This proves Theorem \ref{ch00:intro:thm01}.


\section{Integrated remainders}

In this section we prove Theorem \ref{ch0305:thm003}.
We start by noting that for $\Gamma$ cocompact
the function $N(s,z,w)$ is uniformly bounded in $z,w$
and hence square-integrable
(see e.g.
\cite[Thm. 6.1]{selberg_harmonic_1955})
and \cite[p. 278]{chavel_eigenvalues_1984}).
By Parseval's theorem we get the expansion
\begin{equation}
    G_1(s,z)
    =
    \sum_{t_j>0} \frac{h_s(t_j)^2}{e^s} |\phi_j(z)|^2
    +
    \sum_{t_j\in[0,i/2]} f_s(t_j)^2 |\phi_j(z)|^2,\nonumber
\end{equation}
where $f_s(t_j)$ is defined (recall \eqref{ch02:def:mainterm}) for $t_j\in[0,i/2]$ by
\[
\begin{split}
    f_s(i/2)
    &=
    \frac{1}{e^{s/2}}\left(h_s(i/2)-\frac{\pi e^s}{\vol(\Gamma\backslash\HH)}\right),
    \\
    f_s(t_j)
    &=
    \frac{1}{e^{s/2}}\left(h_s(t_j)-\sqrt{\pi} \frac{\Gamma(|t_j|)}{\Gamma(3/2+|t_j|)} e^{s(1/2+|t_j|)}\right),\;\;t_j\in(0,i/2),
    \\
    f_s(0)
    &=
    \frac{1}{e^{s/2}}\left(h_s(0)-4\big(s+2(\log 2-1)\big)e^{s/2}\right),
\end{split}
\]
and it satisfies $f_s(t_j)=O(e^{-\eps_\Gamma s})$ for every $t_j\in[0,i/2]$, for some $\eps_\Gamma>0$
(this follows from the discussion on the small eigenvalues in section \ref{ch02:subsubsection03}).
We can therefore write
\begin{equation}\label{ch0305:eq011}
    G_1(s,z)
    =
    C_1
    \!+\!
    2 \Re \left(
    \sum_{0<t_j\leq X} \!\! \frac{\pi\,\Gamma(it_j)^2|\phi_j(z)|^2}{\Gamma(3/2+it_j)^2} e^{2it_js}
    \right)
    \!+\!
    O\left(e^{-\eps_\Gamma s} \!+\! \frac{1}{X}\right)
\end{equation}
with
\begin{equation}
    C_1 = 2\pi \sum_{t_j>0} \frac{|\Gamma(it_j)|^2}{|\Gamma(3/2+it_j)|^2}\,|\phi_j(z)|^2.\nonumber
\end{equation}
The coefficients in \eqref{ch0305:eq011} satisfy (by \eqref{0707:eq001} and \eqref{crucial_ineq})
\begin{equation}
    \sum_{T\leq 2t_j \leq T+1} \frac{|\Gamma(it_j)|^2}{|\Gamma(3/2+it_j)|^2}|\phi_j(z)|^2
    \ll
    \frac{1}{T^2}.\nonumber
\end{equation}
In particular this means that $G_1(s,z)$ is bounded in $s$, because
\begin{equation}
    |G_1(s,z)| \ll C_1 + \sum_{t_j>0} \frac{|\phi_j(z)|^2}{t_j^3} + e^{-\eps_\Gamma s} \ll_z 1.\nonumber
\end{equation}
The function $G_1-C_1$
is of the form \eqref{0308:eq001}, satisfies \eqref{ch0301:eq004} with $\beta=2$,
and, choosing $X(Y)=e^Y$, it satisfies \eqref{ch0301:eq003} for every $p\geq 1$.
Hence we can apply Theorem \ref{ch0304:thm01} and infer the existence of all the moments,
and Theorem \ref{ch0303:thm01} to infer the existence of a limiting distribution $\tilde\mu_1$.
By Corollary \ref{apf01:cor01},
the moments of $G_1-C_1$ coincide with the moments of $\tilde\mu_1$.
Shifting $G_1-C_1$ and the measure $\tilde\mu_1$ by adding back $C_1$
we obtain the Theorem for $G_1$.

Consider now the function $G_2(s)$. In this case we have the expansion
\begin{equation}
    G_2(s)
    =
    \sum_{t_j>0} \frac{h_s(t_j)^2}{e^s}
    +
    \sum_{t_j\in[0,i/2]} f_s(t_j)^2,\nonumber
\end{equation}
and therefore we can write
\begin{equation}
    G_2(s)
    =
    C_2
    +
    2 \Re \left(
    \sum_{0<t_j\leq X} \frac{\pi\,\Gamma(it_j)^2}{\Gamma(3/2+it_j)^2} e^{2it_js}
    \right)
    +
    O\left(e^{-\eps_\Gamma s} + \frac{1}{X}\right)\nonumber
\end{equation}
with
\begin{equation}
    C_2
    =
    2\pi \sum_{t_j>0} \frac{|\Gamma(it_j)|^2}{|\Gamma(3/2+it_j)|^2}.\nonumber
\end{equation}
Using this time the estimate \cite[Th. 7.3]{venkov_spectral_1990}
on the distribution of the eigenvalues
\begin{equation}
    \sum_{T\leq 2t_j\leq T+1} 1 \ll T\nonumber
\end{equation}
we can write
\begin{equation}
    \sum_{T\leq 2t_j \leq T+1} \frac{\Gamma(it_j)^2}{\Gamma(3/2+it_j)^2}
    \ll
    \frac{1}{T^2}\nonumber
\end{equation}
and choosing again $X=e^Y$ we see that $G_2-C_2$ is of the form \eqref{0308:eq001},
satisfies \eqref{ch0301:eq004} with $\beta=2$, and  \eqref{ch0301:eq003} for every $p\geq 1$.
Applying Theorem \ref{ch0304:thm01}, Theorem \ref{ch0303:thm01}, and Corollary \ref{apf01:cor01},
we conclude the proof for $G_2$.

Finally consider the function $G_3(s,z)$.
The function $e^{-s/2}\int_0^sN(x,z,z)dx$ 
is an automorphic kernel associated to the function
$k_s^*(u)=e^{-s/2}\int_0^s k_x(u)dx$, where $k_x(u)=\mathbf{1}_{[0,(\cosh x-1)/2)]}(u)$.
The Selberg--Harish-Chandra transform $h_s^*$ of $k_s^*$ is given by
\[h_s^*(t) = \frac{1}{e^{s/2}}\int_0^s h_x(t)dx\]
and can be analysed with analogous computations to those of section \ref{apf_04_section04}.
We claim that $h^*$ is an admissible function in the pretrace formula.
First observe that for the small eigenvalues we can write
\begin{equation}
    h_s^*(i/2)
    =
    \frac{1}{e^{s/2}} \int_0^s (\pi e^x + O(1))dx
    =
    \pi e^{s/2} + O(se^{-s/2}).\nonumber
\end{equation}
Similarly we have
\begin{equation}
    h_s^*(0)
    =
    8(s+2(\log 2-1)) + O(e^{-s/2}),\nonumber
\end{equation}
and, for $t_j\in(0,i/2)$,
\begin{equation}
    h_s^*(t_j)
    =
    \sqrt{\pi}\frac{\Gamma(|t_j|)}{(1/2+|t_j|)\Gamma(3/2+|t_j|)}e^{s|t_j|}
    + O(e^{-\eps_\Gamma s})\nonumber
\end{equation}
for some $0<\eps_\Gamma<1/4$.
For $t_j$ real and positive we use a representation for $h_x(t)$ that is
more suitable for integration in $x$. The expression can be found in
the proof of \cite[Lemma 2.5]{phillips_circle_1994}, and gives
\begin{equation}
    h_x(t_j) = 2\Re(I(x,t_j)e^{x(1/2+it_j)}),\nonumber
\end{equation}
where
\begin{equation}
    I(x,t_j)
    =
    -2i \int_0^\infty (1-e^{iv})^{1/2} (1-e^{-2x-iv})^{1/2} e^{-t_j v} dv.\nonumber
\end{equation}
Integrating in $x$ we can write
\begin{equation}
    \int_0^s h_x(t_j)dx
    =
    2\Re\Big(\int_0^sI(x,t_j)e^{x(1/2+it_j)}dx\Big).\nonumber
\end{equation}
Moving the contour of integration to two vertical lines above $0$ and $s$
in the complex plane, we obtain
\begin{align}
    &\int_0^s
    I(x,t_j)e^{x(1/2+it_j)}dx\nonumber
    \\
    &=
    2\int_0^\infty\!\!\!\!\int_0^\infty
    (1-e^{iv})^{1/2} (1-e^{-2i\lambda-iv})^{1/2} e^{-t_j v} e^{i\lambda/2-\lambda t_j} dv d\lambda\nonumber
    \\
    &-2 e^{s(1/2+it_j)}\int_0^\infty\!\!\!\!\int_0^\infty
    (1-e^{iv})^{1/2} (1-e^{-2s-2i\lambda-iv})^{1/2} e^{-t_j v} e^{i\lambda/2-\lambda t_j} dv d\lambda.\nonumber
\end{align}
Isolating the oscillation $e^{ist_j}$, and using the bound $|1-e^{iv}|\ll \min(1,v)$
to bound the other terms, we conclude that
\begin{equation}
    h_s^*(t_j)
    =
    2\Re\big(A(t_j)e^{ist_j}\big) + O\left(\frac{1}{t_j^{5/2}}\right),\nonumber
\end{equation}
with
\begin{equation}
    A(t_j)
    =
    -2\int_0^\infty\!\!\!\!\int_0^\infty
    (1-e^{iv})^{1/2} e^{-t_j v} e^{i\lambda/2-\lambda t_j} dv d\lambda
    =
    \frac{i\sqrt{\pi}\,\Gamma(it_j)}{(1/2-it)\Gamma(3/2+it_j)}.\nonumber
\end{equation}
Since $A(t_j)\ll t_j^{-5/2}$, we infer that $h_s^*$ is an admissible test function in
the pretrace formula.
Observing moreover that the main terms
that appear in the integration of the small eigenvalues
correspond to the integration of the terms defining $M(s,z,z)$,
we can write for $X\gg 1$
\begin{equation}
    G_3(s,z)
    =
    2\Re\Big(\sum_{0<t_j\leq X} A(t_j)|\phi_j(z)|^2 e^{ist_j}\Big)
    +
    O\left(e^{-\eps_\Gamma s} + \frac{1}{X^{1/2}}\right).\nonumber
\end{equation}
This shows that $G_3(s,z)$ is of the form \eqref{0308:eq001},
and its coefficients $A(t_j)|\phi_j(z)|^2$ satisfy \eqref{ch0301:eq004}
with $\beta=3/2$. Choosing $X=e^Y$,
$G_3(s,z)$ satisfies \eqref{ch0301:eq003} for every $p\geq 1$.
Applying Theorem \ref{ch0304:thm01}, Theorem \ref{ch0303:thm01}, and Corollary \ref{apf01:cor01},
we conclude the proof for~$G_3$.


\section*{Acknowledgments}
I would like to thank Morten S. Risager and Yiannis Petridis for their precious comments
on earlier versions of the paper. Thanks also go to the anonymus referee, whose comments helped
to improve significantly the exposition of the paper.
This work was supported by a Sapere Aude grant from The Danish Council for Independent Research (Grant-id:0602-02161B).


\end{document}